\documentclass[a4paper, 11pt]{article}
\usepackage{anysize}
\marginsize{3,5cm}{2,5cm}{2,5cm}{2,5cm}
\usepackage{amssymb}
\usepackage{amsmath,amsbsy,amssymb,amscd}
\usepackage{t1enc}
\usepackage[cp1250]{inputenc}
\usepackage[british]{babel}
\usepackage[all]{xy}
\usepackage{color}
\usepackage{amsfonts}
\usepackage{latexsym}
\usepackage{amsthm}
\usepackage{mathrsfs}
\usepackage{hyperref}
\usepackage{graphicx}

\usepackage{color}
\usepackage{caption}
\usepackage{subcaption}
\usepackage{enumerate}

\definecolor{orange}{rgb}{1,0.5,0}

\usepackage{mathrsfs} 

\DeclareMathAlphabet{\mathpzc}{OT1}{pzc}{L}{it} 

\def\vphi{\varphi}
\def\a{\alpha}

\newtheorem{definition}{Definition}[section]

\newtheorem{proposition}[definition]{Proposition}
\newtheorem{theorem}{Theorem}
\newtheorem{sublemma}{Sublemma}
\newtheorem{corollary}[definition]{Corollary}
\newtheorem{remark}[definition]{Remark}
\newtheorem{lemma}[definition]{Lemma}

\def\geq{\geqslant}
\def\leq{\leqslant}
\def\R{\mathbb{R}}
\def\T{\mathbb{T}}

\def\Z{\mathbb{Z}}
\def\N{\mathbb{N}}

\def\cB{\mathbb{B}}

\def\th{\theta}
\def\ve{\varepsilon}
\def\cC{\mathcal{C}}
\def\Aut{\operatorname{Aut}}

\newcommand{\bea}{\begin{eqnarray}}
  \newcommand{\eea}{\end{eqnarray}}
  \newcommand{\beab}{\begin{eqnarray*}}
  \newcommand{\eeab}{\end{eqnarray*}}
\renewcommand{\a}{\alpha}
  \newcommand{\be}{\begin{equation}}
  \newcommand{\ee}{\end{equation}}
  
\newcommand{\cI}{\mathcal I}
\newcommand{\cD}{\mathcal D}

\newcommand{\set}[1]{\left\lbrace #1 \right\rbrace}
\newcommand{\abs}[1]{\left| #1 \right|}
\newcommand{\mc}{\mathcal}

\title{On the non-equivalence of the Bernoulli and $K$ properties in dimension four}
\author{Adam Kanigowski \and Federico Rodriguez Hertz\footnote{F. R. H. was supported by NSF grants DMS 1201326 and DMS 1500947}\and Kurt Vinhage\footnote{K. V. was supported by the National Science Foundation under Award DMS 1604796} }
\date{}
\begin{document}
\baselineskip=14pt \maketitle

\vspace{-.25cm}\hfill {\it \small Dedicated to the memory of Roy Adler.}\vspace{.25cm}

\begin{abstract}
We study skew products where the base is a hyperbolic automorphism of $\T^2$, the fiber is a smooth area preserving flow on $\T^2$ with one fixed point (of high degeneracy) and the skewing function is a smooth non coboundary with non-zero integral. The fiber dynamics can be represented as a special flow over an irrational rotation and a roof function with one power singularity. We show that for a full measure set of rotations the corresponding skew product is $K$ and not Bernoulli. As a consequence we get a natural class of volume-preserving diffeomorphisms of $\T^4$ which are $K$ and not Bernoulli.   
\end{abstract}
\section{Introduction}
Bernoulli shifts provide the simplest model for  {\it chaotic behaviour} of dynamical systems in the measurable category. Ornstein showed in \cite{ornstein1} that the entropy of a Bernoulli system is the only invariant up to measurable conjugacy. Another natural property describing chaotic systems is the {\it Kolmogorov property}, or {\it $K$-property} for short. It follows by \cite{roh-sin} that the $K$-property is equivalent to {\it completely positive} entropy: each non-trivial factor of the system has positive entropy. It is an easy observation that Bernoulli shifts are $K$. Kolmogorov conjectured that the converse is also true. This was shown to be false by Ornstein, who contructed a $K$-automorphism which is not Bernoulli \cite{ornstein2}. Ornstein's contruction, however, was done through a certain specialized combinatorial procedure. In particular, it did not indicate whether one should expect equivalence of the properties in any category. Since the introduction of these properties, much effort has been put into understanding how often to expect equivalence of these properties, and finding more natural examples: both in the sense of finding {\it families} of examples in the measurable category
, and finding examples in the smooth category.

A natural place to search for such transformations is {\it skew products} over Bernoulli shifts. Recall that for $T:(X,\cB,\mu)\to (X,\cB,\mu)$, $S:(Y,\cC,\nu)\to (Y,\cC,\nu)$ and $\vphi:X\to \Z$,  the skew product over the {\em base} $T$ with {\em fiber} $S$ and with  {\em cocycle} $\vphi$ is given by
$$
T^{S}_\vphi(x,y)=(Tx, S^{\vphi(x)}y).
$$

$T^S_\vphi$ is sometimes called the {\it total map}. Ornstein, Adler and Weiss proposed to consider a special case of the skew product construction as follows (see, eg, \cite[p.394]{kalikow}, \cite[p. 682 (2)]{weiss}). Let $T=S$ to be a Bernoulli shift on $\{0,1\}^\Z$ with $\mu=(1/2,1/2)^\Z$ and $\vphi(x)= (-1)^{x_0}$ ($x_0$ denotes the zeroth cooridnate of $x$). Such transformations are called $(T,T^{-1})$ transformations or {\it random walks in random environment}\footnote{Sometimes also called random walks in random scenery}. It is almost immediate that the $(T,T^{-1})$ transformation is $K$. 

Notice that one can take the fiber to be a flow $(S_t)$, the cocycle $\vphi$ to take values in $\R$ and obtain more general skew products of the form 
\begin{equation}\label{sk}
T^{(S_t)}_\vphi(x,y)=(Tx,S_{\vphi(x)}(y)).
\end{equation}

Since every positive entropy system has a Bernoulli factor of the same entropy \cite{sinai64}, the skew-product construction is in some sense universal in the measurable category. However, to include all possible transformations, one in general must take cocycles $\vphi : X \to \Aut(Y)$ for some Lebesgue space $Y$. In the skew-product setup, one may analyze the entropy through looking at the contributions from the base and fiber. Indeed, the Abramov-Rokhlin entropy formula \cite{abramov-rohlin} shows that $h\left(T^{S_t}_\vphi\right) = h(T) + h(S)\left|\int \vphi\right|$. This observation also shows that if the Bernoulli base is also a Bernoulli factor with maximal entropy, the fiber entropy must be 0. This can occur for transformations of type \eqref{sk} in two ways: by having 0 average for the cocycle $\vphi$ (as in the $(T,T^{-1})$ examples), or by having a 0 entropy fiber $S$. One may, of course, take positive entropy fibers with $\int \vphi > 0$, but the Bernoulli base will fail to be a candidate for isomorphism. In \cite{feldman}, J. Feldman introduced the notion of loosely Bernoulli process to build skew-product examples of K but not Bernoulli automorphisms. His key criterion is that if $T$ is loosely Bernoulli then $S$ should be loosely Bernoulli. In this case he used as base dynamics the shift on two symbols with equal mass and $\phi$ is the $0, 1$ function, so the integral is positive. The fiber map $S$ is what is now called a {\it standard} map: a 0 entropy, loosely Bernoulli system.  Later, Kalikow used the loosely Bernoulli property to show that the $(T,T^{-1})$-transformation is $K$ but not Bernoulli in \cite{kalikow}.

If one wishes to exhibit examples in the smooth category along the same construction, one can easily replace the Bernoulli base with a smooth realization: toral automorphisms. However, since the cocycle $\vphi$ must also be smooth, one must take the fiber system as a smooth flow (and not transformation)\footnote{If the fiber is a transformation, $\varphi$ must take values in $\Z$ and hence must be locally constant}. Such examples were first obtained by Katok in \cite{Kat} using a skew product of the form \eqref{sk}. There, $T$ is an Anosov diffeomoprhism, $S_t$ is a smooth ergodic flow, and $\vphi$ is a positive cocycle which is not cohomologous to a contant.
 The principle achievement of \cite{Kat} was showing that the skew product is $K$ with a general cocycle $\vphi$. The Feldman criterion was then used to produce first \textbf{smooth} examples of $K$ systems which are not Bernoulli: if $(S_t)$ is not {\em loosely Bernoulli}, the skew product is not Bernoulli. One may take, for example, $S_t=h_t\times h_t$ where $h_t$ is the horocycle flow on the unit tangent bundle to a surface of constant negative curvature\footnote{The cartesian square of the horocycle flow was shown to be not loosely Bernoulli in \cite{ratner}.}. 
This example has the advantage of having very explicit formulas, but also the disadvantage of being 8-dimensional. While not written formally at the time, it was remarked that one could obtain an example in dimension 5. This example can be constructed in the following way: The {\it approximation-by-conjugation} method developed by Anosov and Katok can produce smooth realizations of certain models in the measurable category \cite{anosov-katok}. This yields a non-loosely Bernoulli transformation $S : \T^2 \to \T^2$, following the combinatorial construction of Feldman \cite{feldman} (see \cite{kat-constructions} or \cite{benhenda} for a modern treatment of this construction). Then taking the constant-time suspension of this transformation yields a non-loosely Bernoulli constant-time suspension in the fiber. Then the skew-product and the direct product of the special flow with base $T$ and roof $\vphi$ and $S$ are sections of the same smooth flow, and hence Kakutani equivalent. As a result, $T^{(S_t)}_\vphi$ is not loosely Bernoulli, and is $K$ by \cite{Kat}.

Later Rudolph using methods similar to the Kalikow $(T,T^{-1})$ example, obtained other smooth examples in dimension $5$ \cite{rudolph}. In the class considered in \cite{rudolph}, one can take 
$(S_t)$ to be the geodesic flow and $\vphi$ some smooth cocycle which is not a coboundary such that $\int \vphi = 0$. 
 One the other hand, by Pesin theory, smooth $K$ automorphisms in dimension $2$ are Bernoulli \cite{pesin}. 

It is therefore natural to ask about the equivalence of $K$ and Bernoulli properties in dimensions $3$ and $4$.
The case of dimension $3$ is more difficult, as skew-product constructions respecting the smooth structure would require a circle fiber. Since we assume measure-preserving, the action on fibers must be isometric. Such skew-product extensions are always Bernoulli \cite{burton-shields}.  On the other hand, in the context of volume preserving partially hyperbolic automorphisms on $3$ dimensional manifolds with central foliation by circles, the accessibility property implies the K-property (see \cite{HHU} for definitions). Among these systems, if central exponent is non-zero, by Y. Pesin \cite{pesin} the system is Bernoulli. If the central exponent is zero but the central foliation is absolutely continuous, then the system is again Bernoulli by an application of the results in \cite{AVW} and the aforementioned result of \cite{burton-shields}. The remaining case is therefore the 3 dimensional partially hyperbolic, volume preserving systems, with accessibility property, non-absolutely continuous compact central foliation and zero central exponent. G. Ponce, A. Tahzibi and R. Var\~{a}o conjectured later that not all such systems are Bernoulli \cite{ponce-tahzibi-varao}.
 
We summarize the main developments for examples of $K$, non-Bernoulli systems here, in chronological order:

\begin{center}
\begin{tabular}{r|c|c|c|c}
                     & LB Fiber & Fiber Entropy & Smooth & $\int \vphi$ \\
\hline
Ornstein \cite{ornstein2}  & N/A & N/A & No & N/A\\
\hline
Feldman \cite{feldman}  & No & 0 & No & $\not= 0$ \\
\hline
Katok \cite{Kat}   & No & 0 & Yes & $\not= 0$ \\
\hline
Burton \cite{burton}& Yes & Any & No & $\not= 0$ \\
\hline
Kalikow \cite{kalikow}   & Yes & $> 0$ & No & $0$ \\
\hline
Rudolph \cite{rudolph} & Yes & $> 0$ & Yes & $0$ \\
\hline
Theorem \ref{thm:main} & Yes & 0 & Yes & $\not= 0$
\end{tabular}
\end{center}

Another interesting property to investigate is whether the examples are loosely Bernoulli. Following \cite{feldman}, most proofs actually prove that the skew product is not loosely Bernoulli. In fact, the only example above which does not show that the skew product is not loosely Bernoulli is \cite{burton}. In those examples and ours, the fiber is loosely Bernoulli. We do not have reasonable evidence either way to conjecture whether the total map is loosely Bernoulli.

\subsection{Statement of Main Results}

We show that there are a class of flows on surfaces $(S_t)$, which when used in transfomations of the form \eqref{sk} over a Bernoulli base, yield $K$, non-Bernoulli examples. In particular, we produce the first such examples in dimension 4. For the fiber, we take a smooth flow $K_t \curvearrowright (\T^2,\lambda)$  with a single singularity with a high order of degeneracy (we will also call this flows Kochergin flows). More precisely, we take a time change of a linear flow with a Hamiltonian which at the point of degeneracy is locally given by $(ax+by)(x^2+y^2)^l$ for some sufficiently large $l$. Such flow can be represented as {\em special flow} over irrational rotation by $\alpha$ and roof function which is $C^2(\T\setminus \{0\})$, $f>0$ and satisfies 
\begin{equation}\label{asu}\lim_{y\to 0^+}\frac{f(y)}{y^{-(1-\eta)}}=M_1\text{   and   }\lim_{y\to 0^-}\frac{f(y)}{y^{-(1-\eta)}}=M_1
\end{equation}

\begin{equation}\label{asu2}\lim_{y\to 0^+}\frac{f{'}(y)}{y^{-(2-\eta)}}=-N_1\text{   and   }\lim_{y\to 0^-}\frac{f{'}(y)}{y^{-(2-\eta)}}=N_1
\end{equation}

\begin{equation}\label{asu3}\lim_{y\to 0^+}\frac{f{''}(y)}{y^{-(3-\eta)}}=R_1\text{   and   }\lim_{y\to 0^-}\frac{f{''}(y)}{y^{-(3-\eta)}}=R_1
\end{equation}
where $\eta$ is a small number (since the order of degeneracy of the saddle is high),  $\eta\in(0,\frac{1}{100})$, $0< M_1,N_1,R_1<+\infty$. For more on the connection between smooth flows on surfaces and special flows over rotations (or more generally, interval exchange transformations), see \cite{kochergin,Fa-Fo-Ka}. We will denote the flow given by $f,\alpha$ by $(K_t^{f,\alpha})$ (or, if $f$ and $\alpha$ are fixed and understood, we will often simply write $(K_t)$). Finally, the cocycle $\vphi:\T^2\to \R$ is a smooth function such that $\vphi_0=\int_{\T^2}\vphi \, d\lambda \not= 0$ and
\begin{equation}\label{eq:K}
\psi(Tx)-\psi(x)=\vphi(x)-\vphi_0,
\end{equation}
has no continuous solutions $\psi$. We say that $\alpha\in \cD$ if and only if there exists a constant $C=C(\alpha)$ such that for any $p,q\in \Z$, $q > 0$,
$$
\left|\alpha-\frac{p}{q}\right|\geq \frac{C}{q^2\log^{11/10}q}.
$$
It is classical that the set $\cD$ has full Lebesgue measure on $\T$ (see, eg, \cite{khinchin}).
With the above notation, our main result is the following:
\begin{theorem}\label{thm:main}Let $\alpha\in \cD$. Then the automorphism $A^{(K_t)}_{\vphi}$ is $K$ and is not Bernoulli.
\end{theorem}
By the results in \cite{Kat}, since \eqref{eq:K} has no continuous solution, $A^{(K_t)}_{\vphi}$ is $K$. Here, since the base dynamics is a toral automorphism there is no need to assume the flow $K_t$ is weak mixing (though it is the case for Kochergin flows, \cite{kochergin}). Therefore to prove Theorem \ref{thm:main} one has to show that $A^{(K_t)}_{\vphi}$ is not Bernoulli. We will in fact show that $A^{(K_t)}_\vphi$ is not VWB (very weak Bernoulli, see Definition \ref{def:VWB-general}) with respect to some convenient partition. 

\subsection{Some Observations on Fiber Complexity}

Let us make a few remarks highlighting the differences between our approach and those of \cite{kalikow}, \cite{rudolph} or \cite{Kat}. First, thanks to Pesin formula \cite{pesin}, smooth surface flows always have $0$ entropy. In fact, the orbit growth is {\em polynomial} and therefore the methods of \cite{kalikow}, \cite{rudolph} showing the non-Bernoulli property will fail for us. On the other hand every smooth surface flow is standard and therefore has a section on which the first return map is an {\em interval exchange transformation} (which is loosely Bernoulli, see \cite{katok-sataev}). 
In summary, our example is the first which uses a smooth standard fiber transformation. Such a fiber transformation was constructed by \cite{burton} in the measurable category. However, like the original example of Ornstein, it was created exclusively for the purpose of having a loosely Bernoulli fiber, and did not arise naturally.

Let us also emphasize another important consequence of Theorem \ref{thm:main}. It was suggested in \cite{weiss}, to  consider skew-products over Bernoulli transfomrations, and study which fiber transformations result in a Bernoulli skew product. With {\it elliptic} fibers (very slow or no orbit growth), one expects the skew product to remain Bernoulli. This is shown to be true  in \cite{adler-shields}, where the fiber is an irrational rotation and in \cite{burton-shieldsMixing}, where the fiber is a mixing rank one system with slow orbit growth. 
With slightly higher complexity, namely weakly mixing systems which admit good cyclic approximations, one often finds non-Bernoulli extensions. This was the original method of Feldman and Katok. Benhenda was able to use these methods to find uncountable class of non-isomorphic extensions \cite{benhenda}.
From \cite{kalikow}, \cite{rudolph} one can deduce that if the fiber is {\it hyperbolic} (i.e., has exponential orbit growth or more explicitly, positive entropy) and $\int \vphi = 0$, then the corresponding skew product does not remain Bernoulli.\footnote{If $\int \vphi \not= 0$, then the base is not the maximal Bernoulli factor.} This is also confirmed in a recent work of Austin \cite{austin}, where another uncountable family of non-isomorphic skew product with Bernoulli base is constructed.  
 
However, nothing was known for fibers with polynomial or intermediate growth, i.e. {\it parabolic} fibers. Kochergin flows are parabolic in this sense, so one may consider Theorem \ref{thm:main} as a first step towards the study of skew products with parabolic fibers. It also suggests that such skew products do not remain Bernoulli. The methods used to prove Theorem \ref{thm:main} use tools with properties that generalize readily. The authors plan to continue the investigation of skew products with parabolic fibers in a subsequent paper, in particular the case of horocycle flows, their time changes and time changes of nilflows.

\subsection{Plan of the paper.}
In Section \ref{sec.def} we give definitions of very weak Bernoulli (VWB), Kochergin special flows and give some Denjoy-Koksma estimates for a Kochergin roof function. In Section \ref{sec.vwb} we first give a characterization of (VWB) for zero entropy extensions of Bernoulli systems (see Proposition \ref{prop:skew-vwb-def}) then we state Theorem \ref{thm:vwb} which contains the main combinatorial properties. Finally in Subsection \ref{main.vwb} we use Proposition \ref{prop:skew-vwb-def} to show that Theorem \ref{thm:vwb} implies Theorem \ref{thm:main}. 

A proof of Theorem \ref{thm:vwb} is then given in Section \ref{vwb.cond} conditionally on Proposition \ref{techn}. The rest of the paper is devoted for proving 1-6 in Proposition \ref{techn}. Section \ref{bcy} is devoted for construction of $B$ and   
$\{C_y\}_{y\in M}$ from Proposition \ref{techn}. The set $B$ is constructed by Egorov's theorem type of reasoning as we want good control on $\vphi$ (given by CLT) and on $f$ (not coming to close to singularity). For $y\in M$, $C_y$ is the set of points whose orbit do not come close to $y$  in a short time. Properties 1. and 2. will follow then automatically by the construction. Property 5. is a striaghtforward consequence of Lemma \ref{cons:dis} (which is a consequence of diophantine assumptions). The construction of $C_y$ gives then Properties 3. and 6. (see Subsection \ref{p36}). 

The most difficult is 4., which is handeled separately in Section \ref{sec.4}. It requires vertical stretch for nearby points. This is guaranteed by $f'_{n}$ being of order $n^{2-\eta}$. This does not happen for all times, since we have cancelations (the roof is symmetric). Proposition \ref{mpr} however shows that $f'_n$ is of correct order for most of the times (with a polynomial gain). Finally in Section \ref{p:mpr} using probabilistic tools, we prove Proposition \ref{mpr}.

\subsection*{Acknowledgments}
The authors would like to thank Mariusz Lema\'nczyk and Jean-Paul Thouvenot for several discussions and suggestions on the subject. The authors would also like to thank Anatole Katok for discussions and several suggestions to improve the paper.

\section{Kochergin special flows and Denjoy-Koksma estimates}\label{sec.def}

\subsection{Special flows}
Let $R_\a:\T\to\T$, $R_\a(y)=y+\a \; \text{\rm mod }1$, where $\a\in \T$ is an irrational number with the sequence of denominators $(q_n)_{n=1}^{+\infty}$ and let $\psi\in L^1(\T,\mathcal{B},\lambda)$ be a strictly positive function. We recall that the special flow $T_t:=T_t^{\a,\psi}$ constructed above $R_\a$ and under $\psi$ is given by 
\begin{eqnarray*}
\T \times \R / \sim  &  \rightarrow &  \T \times \R / \sim  \\
 (y,s) & \rightarrow & (y,s+t), \end{eqnarray*}
where $\sim$ is the
identification 
\begin{equation}
\label{FlowSpace}
(y, s + \psi(y)) \sim (R_\a(y),s) \,.
\end{equation}
Equivalently, this special flow  is defined for  $t+s \geq 0$ (with a similar definition for negative times) by 
$$T_t(y,s) = (y+N(y,s,t)\a , t+s-   \psi_{N(y,s,t)} (y))$$ 
   where $N(y,s,t)$ is the unique integer such that 
\begin{equation}
\label{D-C}
0 \leq  t+s-   \psi_{N(y,s,t)} (y) \leq \psi(y+N(y,s,t)\a),\end{equation}    
and
$$
\psi_n(y)=\left\{\begin{array}{ccc}
\psi(y)+\ldots+\psi(R_\alpha^{n-1}y) &\mbox{if} & n>0\\
0&\mbox{if}& n=0\\
-(\psi(R_\alpha^ny)+\ldots+\psi(R_\alpha^{-1}y))&\mbox{if} &n<0.\end{array}\right.$$

\paragraph{Kochergin flows under consideration}
Flows which we will consider are special flows over an irrational rotation $\a$, and a roof function $f$ satisfying \eqref{asu}, \eqref{asu2}, \eqref{asu3} with $\eta\in (0,1/100)$. To simplify notation we assume that $M_1=N_1=R_1=1$ and that $\int_\T f\,d\lambda=1$. We moreover assume that $\alpha\in \cD$, i.e. $\alpha$ is not to well aproximated by rationals. We will denote the space on which the Kochergin flow $(K_t)_{t\in \R}$ acts by $(M,\mu)$. 
Notice that we have the following metric on $M$:
$$
d((y,s)(y',s')):=d_H(y,y')+|s-s'|,
$$
where $d_H(y,y')=\|y-y'\|$ is the horizontal distance. For simplicity we will often denote $(y,s)$ and $N(y,s,t)$ respectively by $y$ and $N(y,t)$. For a set  $W\subset \T$ we denote $W^f:=\{(y,s)\in M\;:\; y\in W\}$.




\subsubsection{Denjoy-Koksma inequalities}
The following lemma is a consequence of Denjoy-Koksma inequality, (see for example \cite{Fa-Fo-Ka}, Lemma 3.1.).
\begin{lemma}\label{koksi} For every $y\in \T$ and every $ |M|\in [q_s,q_{s+1}]$  we have 
\begin{equation}\label{koks0}
f(y^M_{min})+\frac{q_{s}}{3}\leq f_M(y) \leq f(y^M_{min})+3q_{s+1}
\end{equation} 
\begin{equation}\label{koks1}
f'(y^M_{min})-8q_{s+1}^{2-\eta}<
|f'_M(y)|<f'(y^M_{min})+8q_{s+1}^{2-\eta}
\end{equation}
and
\begin{equation}\label{koks2}
f''(y^M_{min})\leq f''_M(y)<
f''(y^M_{min})+8q_{s+1}^{3-\eta},
\end{equation}
where $y^M_{min}=\min_{0\leq j\leq M}d(y+j\a,0)$.
\end{lemma}

\section{Very weak Bernoulli for skew products and a theorem which implies Theorem \ref{thm:main}}\label{sec.vwb}

Our strategy for showing that a transformation $T$ is not Bernoulli is to disprove the {\it very weak Bernoulli property} with respect to a convenient partition $\mathcal R$. The definition has the advantage that if $T$ is Bernoulli, then it has the very weak Bernoulli property for every partition. The results of this section follow \cite{burton-shields} and \cite{shields}. Let $T : (X, \mu) \to (X, \mu)$ be a measurable transformation preserving a probability measure $\mu$, and $\mathcal R$ be a finite partition of $X$. The $\mathcal R$-name of $x$ is the sequence of atoms of $x$ which the orbit of $x$ determines, denoted by $x_i$. Note that $\mu$ induces a measure on $\Sigma_\mathcal R^- = \set{ (\dots,r_{-2},r_{-1},r_0) : r_i \in \mathcal R}$ by identifying cylinder sets with the corresponding elements of $\displaystyle \bigvee_{i=0}^\infty T^i\mathcal R$.

Let us define the very weak Bernoulli property following \cite[p. 134]{shields}.

\begin{definition}
\label{def:VWB-general}
$T$ is {\it very weak Bernoulli} (VWB) with respect to $\mathcal R$ if for every $\ve > 0$, there exists some $N \in \N$ and a measurable set $G \subset \displaystyle \bigvee_{i=0}^\infty T^i\mathcal R$ (meaning it is measurable with respect to this partition) such that $\mu(G) > 1-\ve$ and for every pair of atoms $r,\bar{r} \subset G$ of $\bigvee_{i=0}^\infty T^i\mc R$, there is a $\mu$-preserving map $\Phi_{r,\bar{r}} : r \to \bar r$ and a set $L \subset  r$ such that:

\begin{enumerate}[(i)]
\item $\mu^r(L) < \ve$
\item If $x \not\in L, x\in r$, $\#\set{i \in [1,N] : \Phi(x)_i = x_i} \ge (1-\ve) N$
\end{enumerate}
Here $\mu^r$ stands for the conditional measure of $\mu$ along the atom $r$ of the partition $ \displaystyle \bigvee_{i=0}^\infty T^i\mathcal R$.
\end{definition}

For the benefit of the reader to unwrap the definitions, we point out the structures above in the case of a Bernoulli transformation with the parition into legnth 1 cylinders. Then points lie in the same atom of $\bigvee_{i=0}^\infty T^i\mc R$ if and only if they share the same present and past. Then we may take $G$ to be the full shift. Then given two atoms $r,\bar{r}$ (ie, two pasts), there is an obvious map which simply re-assigns the pasts and presents of points, match not only most codes of the orbit, but all of them. So we interpret the VWB condition as the future having aribtrarily small dependence on the past and present.

\subsection{Zero Entropy Extensions of Bernoulli Systems}
Definition \ref{def:VWB-general} can be difficult to work with in general, but in the case of skew products, can be simplified by making a few assumptions. Suppose that $S$ is a Bernoulli transformation, so that $S$ is the shift on $\Sigma_d$ with probabilities $(p_1,\dots,p_d)$, $\varphi : \Sigma_d \to \R$ is a measurable function, and $(K_t) : M \to M$ is a zero entropy flow preserving a measure $\nu$ on $M$.  We now assume that $X = \Sigma_d \times M$ and $T : X \to X$ takes the following form:

\begin{equation}
\label{eq:bern-skew}
T(x,y) = (S(x),K_{\varphi(x)}(y))
\end{equation}

Let $P=P_d$ be the partition of $\Sigma_d$ into cylinders $[i]_0$, $i=1,\dots, d$ and let $Q$ be a finite partition of $M$ such that $P_d\times Q$ is generating for $T$ (see Lemma \ref{generatingpartition}). Here we denote with $P_d\times Q$ the partition into products of elements in $P_d$ and $Q$. Note also that $\Sigma_d = \Sigma_d^- \times \Sigma_d^+$, and $\mu_d = \mu_d^- \times \mu_d^+$, where $\Sigma_d^- = \set{(\dots,x_{-2},x_{-1},x_0) }$ and $\Sigma_d^+ = \set{(x_1,x_2,\dots)}$. The following proposition is adapted from \cite{burton-shields}, where it was phrased only for the case when $d = 2$ and the $(1/2,1/2)$ measure on $\Sigma_2$ (although our deduction of this equivalence is virtually identical):

\begin{proposition}
\label{prop:skew-vwb-def}
Assume $T$ takes the form \eqref{eq:bern-skew} with $K$ of zero entropy and that $Q$ is a partition of $M$ such that $P_d\times Q$ is generating for  $T$. Then $T$ is very weak Bernoulli with respect to $P_d\times Q$ if and only if for every $\ve > 0$, there exists a set $G \subset \Sigma_d^- \times M$ with $\mu_d^- \times \nu(G) > 1-\ve$ such that if $(x^-,y),(\bar{x}^-,\bar{y}) \in G$, there exists $\Phi_{(x^-,y),(\bar{x}^-,\bar{y})} : \Sigma_d^+ \to \Sigma_d^+$ preserving $\mu_d^+$ and a set $L\subset \Sigma_d^+$ such that:

\begin{enumerate}[(i)]
\item $\mu_d^+(L) < \ve$
\item $\#\left\{\begin{array}{c}i \in [1,N] : T^i(x^-,x^+,y)\;\;\mbox{and}\;\;T^i (\bar{x}^-,\Phi(x^+),\bar{y})\\
\mbox{are in the same}\; P_d\times Q\;\mbox{atom}\end{array}\right\} \ge (1-\ve) N$ if $x^+ \not\in L$.
\end{enumerate}
\end{proposition}

The following is analogous to \cite[Lemma 1]{burton-shields} and will be used to prove Proposition \ref{prop:skew-vwb-def}.
\begin{lemma}
\label{lemma:bslemma1}
Assume $T$ takes the form (\ref{eq:bern-skew}), where $K_t$ is of zero entropy and let $Q$ be a partition such that $P_d\times Q$ is generating for $T$. Then there is a set of full measure $\tilde X\subset X$ such that if $r$ is an atom of the past partition $\displaystyle \bigvee_{i=0}^\infty T^i(P_d\times Q)\cap \tilde X$ then there is a sequence $s^-\in \Sigma_d^-$ and $y_0\in M$ such that $$r=\{(x,y)\;:\;x^-=s^-, y = y_0\}=\left(\{s^-\}\times\Sigma_d^+\right)\times\{y_0\}$$ i.e. the second coordinate of the atom $p$ is determined.
\end{lemma}
\begin{proof}[Proof of Lemma \ref{lemma:bslemma1} ]
On one hand we know that $h(S)=h(T)$. On the other hand, by Abramov-Rokhlin formula we have that $h(T)=h(S)+h(T|\mathcal F_M)$ where $h(T|\mathcal F_M)$ is the entropy condioned to the $M$-fibers. To be more precise for a.e. $x\in \Sigma_d$ we have that $$h(T|\mathcal F_M)=H\left(\{x\}\times Q\left|\displaystyle \bigvee_{i=1}^\infty T^i(\{S^{-i}(x)\}\times Q)\right)\right.$$
Since $h(T|\mathcal F_M)=0$ we get that $$\{x\}\times Q \succ \displaystyle \bigvee_{i=1}^\infty T^i(\{S^{-i}(x)\}\times Q)$$ for almost every $x$. This implies that $$\displaystyle \bigvee_{i=0}^\infty T^i(\{S^{-i}(x)\}\times Q)=\displaystyle \bigvee_{i=-\infty}^\infty T^i(\{S^{-i}(x)\}\times Q)$$ for almost every $x$. Since $P_d\times Q$ is generating we have that the right hand side is a single point and the proof follows.
\end{proof}

\begin{proof}[Proof of Proposition \ref{prop:skew-vwb-def}]
From Lemma \ref{lemma:bslemma1} we get that restricted to the set of full measure $\tilde X$, the atoms of the partition $\displaystyle \bigvee_{i=0}^\infty T^i(P_d\times Q)$ are identified naturally with $\Sigma_d^+$. With this identification, Proposition \ref{prop:skew-vwb-def} follows immediately from the definition of VWB. 
\end{proof}

The next lemma shows that a finite partition $Q$ such that $P_d\times Q$ is generating for $T$ always exists. 
\begin{lemma}\label{generatingpartition}
Assume $T$ is ergodic, $K$ has finite entropy and $\varphi\in L^\infty$. Then there is a finite partition $Q$ of $M$ and a subset $\tilde X$ of $X$ of full measure such that $(P_d\times Q)\cap \tilde X$ is a generating partition for $T$.
\end{lemma}

Observe that in our situation $K$ is of zero entropy. The proof of this Proposition is contained in the Appendix.

\subsection{Smooth Dynamics and Bernoulli Shifts}
\label{sec:reduce-to-bern}

In this section, we relate the combinatorial structures above with smooth dynamical systems. Let $A : \T^2 \to \T^2$ be an hyperbolic toral automorphism. Every such automorphism has a Markov partition, which induces an almost one-to-one semiconjugacy $h : \Sigma_A \to \T^2$ taking a Markov measure on the subshift of finite type $\Sigma_A$ to Lebesgue measure on $\T^2$. Since it is almost one-to-one, it is also a measurable conjugacy (not just semiconjugacy, as it is in the topological world). Furthermore, this map is H\"older continuous, so the pullback of any H\"older (in particular, smooth) function on $\T^2$ will remain H\"older in $\Sigma_A$.

Since any subshift of finite type is also measurably isomorphic to a Bernoulli shift, we may further get a measurable isomorphism $h' : \Sigma_d \to \Sigma_A$ taking a Bernoulli measure with weights $(p_1,\dots,p_d)$ to the desired Markov measure on $\Sigma_A$.

\begin{remark}
$h'$ is only measurable, so the pullback of a H\"older continuous function $\varphi$ on $\Sigma_A$ or $\T^2$ may (and most likely will) fail to be H\"older. However, any statistical asymptotics made for a H\"older continuous function on $\Sigma_A$ or $F$ can also be made for its pullback by $h'$, since the dynamics and measures are both intertwined.
\end{remark}

\subsection{A combinatorial result which implies Theorem \ref{thm:main}}\label{main.vwb}
For a partition $Q$ of $M$, $(x,y),(x',y') \in \T^2\times M$ and $N\in \N$ let
\begin{equation}\label{DNQ}
D_N^Q\left((x,y),(x',y')\right):=\frac{\left|\{i\in [0,N-1]\;:\; K_{\vphi_i(x)}(y), K_{\vphi_i(x')}(y') \text{ are in one atom of } Q\}\right| }{N}.
\end{equation}
We will show the following:
\begin{theorem}\label{thm:vwb} There exist a set $B\subset \T^2\times M$, $\mu\times \nu(B)>9/10$, a measurable family of sets $(C_y)_{y\in M}$, $C_y\subset M$, $\nu(C_y)\geq 9/10$ and a partition $Q$ of $M$ such that for every $(x,y)\in B$, $(x',y')\in B\cap \T^2\times C_y$ and 
for every $N\in\N$ 
\begin{equation}\label{eq:ham}
D_N^Q\left((x,y),(x',y')\right)<9/10.
\end{equation}
\end{theorem}

Theorem \ref{thm:vwb} is the most important combinatorial result of the paper. Before we prove Theorem \ref{thm:vwb}, let us show how it implies Theorem \ref{thm:main}.
\begin{proof}[Proof of Theorem \ref{thm:main}]
We begin by using the observations of Section \ref{sec:reduce-to-bern} to write the system $A_\vphi^{(K_t)}$ as the skew product $T$ of a zero-entropy dynamical system over a Bernoulli shift on $\Sigma_d$. We may then pull back the structures of Theorem \ref{thm:vwb} to these more convenient combinatorial structures. To simplify the notation we denote the pullbacked structures by the same symbols as the original ones.

Observe that if a partition $\tilde Q \succ Q$ then $D^{\tilde Q}_N\leq D^Q_N$, so we may refine the partition in Theorem \ref{thm:vwb} with the one in Lemma \ref{generatingpartition} and assume without loss of generality that $P_d\times Q$ is generating for $T$. In particular, given $\ve > 0$, we may find the partition $Q$ of $M$ guaranteed by Theorem \ref{thm:vwb} such that $P_d\times Q$ is generating for $T$. We will show that $T$ fails to be very weak Bernoulli with respect to $Q$ and $\ve=1/100$ using the criteria of Proposition \ref{prop:skew-vwb-def}. We assume otherwise, and let $G\subset\Sigma^-_d\times M $ be the set in Proposition \ref{prop:skew-vwb-def}, $\mu^-_d\times\nu(G)\geq 99/100$. We now work in the set of quadruples $(x^-,y,\bar{x}^-,\bar{y})\in\Sigma^-_d\times M\times\Sigma^-_d\times M$, and let $Z$ denote such a quadruple. Then the conclusion of Proposition \ref{prop:skew-vwb-def} is the existence of a family of measure-preserving functions $\Phi_Z : \Sigma_d^+ \to \Sigma_d^+$ for $Z \in G \times G$ such that $D_N^Q((x^-,x^+,y),(\bar{x}^-,\Phi_Z(x^+),\bar{y})) \geq 99/100$ whenever $x^+$ is in a set $L=L_Z$ with $\mu_d^+(L) \ge 99/100$.
 
Let $$\hat G=\{(x^-,y,\bar x^-,\bar y)\in G\times G\;:\; \bar y\in C_y\}.$$ By Fubini's Theorem we have that $(\mu_d^-\times \nu \times \mu_d^-\times \nu)(\hat G)\geq 88/100$.
 
Now let $B \subset \Sigma_d \times M = \Sigma_d^- \times \Sigma_d^+ \times M$ be the set from Theorem \ref{thm:vwb} of measure $>9/10$. Given $(x^-,y)\in\Sigma^-_d\times M$, let $$B_{(x^-,y)} = \{ x^+ \;:\; (x^-,x^+,y) \in B\} \subset \Sigma_d^+$$ be the ``slice'' of $B$ trough $\Sigma_d^+$. Let $B'\subset \Sigma_d^- \times M $ denote the set of pairs $(x^-,y)$ such that $\mu_d^+(B_{(x^-,y)}) > 8/10.$ Then using Fubini's Theorem again $$\frac{9}{10}\leq(\mu\times\nu)(B)\leq(\mu^-_d\times\nu)(B')+(1-(\mu^-_d\times\nu)(B'))\frac{8}{10},$$ and hence $(\mu^-_d\times\nu)(B')\geq \frac{1}{2}$. So, $(\mu^-_d\times\nu\times\mu^-_d\times\nu)(\hat G\cap (B'\times B'))\geq\frac{13}{100}>0.$

 Pick $Z=(x^-,y,\bar x^-,\bar y)\in \hat G\cap (B'\times B')\subset G\times G$. Let $\Phi_Z : \Sigma_d^+ \to \Sigma_d^+$ such that $D_N^Q((x^-,x^+,y),(\bar{x}^-,\Phi_Z(x^+),\bar{y})) \geq 99/100$ whenever $x^+$ is in a set $L$ with $\mu_d^+(L) \ge 99/100$. Since $(x^-,y,\bar x^-,\bar y)\in B'\times B'$, we have that $\mu_d^+(B_{(x^-,y)}) > 8/10$ and $\mu_d^+(B_{(\bar x^-,\bar y)}) > 8/10$, hence $\mu_d^+(B_{(x^-,y)}\cap L) > 79/100$ and since $\Phi_Z$ is measure preserving $$\mu_d^+(\Phi_Z(B_{(x^-,y)}\cap L))\cap B_{(\bar x^-,\bar y)})\geq 59/100>0.$$ Take $x^+\in B_{(x^-,y)}\cap L$ such that $\bar x^+:=\Phi_Z(x^+)\in B_{(\bar x^-,\bar y)}$, i.e. $(x^-,x^+,y)\in B$ and $(\bar x^-,\bar x^+,\bar y)\in B$. Then, on one hand $$D_N^Q((x^-,x^+,y),(\bar{x}^-,\Phi_Z(x^+),\bar{y})) \geq 99/100$$ since $Z\in G\times G$ and $x^+\in L$ but on the other hand $(x^-,x^+,y)\in B$, $(\bar x^-,\bar x^+,\bar y)\in B$ and $\bar y\in C_y$ (since $Z\in \hat G$), so by Theorem \ref{thm:vwb} $$D_N^Q((x^-,x^+,y),(\bar{x}^-,\bar x^+,\bar{y})) \leq 9/10,$$ a contradiction.  \end{proof}

Therefore it remains to prove Theorem \ref{thm:vwb}.

\section{Discussion and proof of Theorem \ref{thm:vwb}}\label{vwb.cond}

\subsection{Some Explanation}

\subsubsection{Additional Notations}

For $(x,y),(x',y')\in \T^2\times M$ and $n\in\Z$, we denote with  $y_n:=K_{\vphi_n(x)}y$ and for $(x',y')$, $y_n':=K_{\vphi_n(x')}y'$ and define $R_n:=d_H(y_n,y'_n)^{-1}$, observe that $R_0=d_H(y,y')^{-1}$. Given $(x,y),(x',y')\in \T^2\times M$, $\xi>0$ and $N,j\in \N$ denote
\begin{multline}\label{anj}
A_j^{N,\xi}(x,y,x',y'):=\{n\in [0,N]\;:\;d(y_n,y'_n)<\xi
\;\text{ and }\;
2^{j}\leq R_n< 2^{j+1}\}.
\end{multline}
Observe that the second condition is the same as $2^{-j-1}<d_H(y_n,y'_n)\leq 2^{-j}$.

\subsubsection{A Summary of Technical Content}

This section contains the main technical tools for establishing the criteria of Theorem \ref{thm:vwb}. The convenient partition $Q$ is a partition into boxes of small diameter with the neighbourhood of the cusp being one atom.
The first tool is Proposition \ref{techn1}. It shows that for choices of $(x,y),(x',y')$  in $B$ their orbits stay away from the cusp for most of the times (properties a.,b.). (Property c.) says that for these good times they cannot stay at a similar level of closeness for long  (with the level of closeness measured by the sets $A_j^{N,\xi_0}$). In fact we have an exponential estimate (given by $\eta_0$). Summing up over all $j\in \N$ and using c. gives the assertion of Theorem \ref{thm:vwb}. So it remain to prove Proposition \ref{techn1}. This is done by the use of Proposition \ref{techn}, which is also the most technical and involved part of the paper.




\begin{figure}[!ht]
\centering
\begin{minipage}{2.5in}
\centering
\includegraphics[width=2.5in]{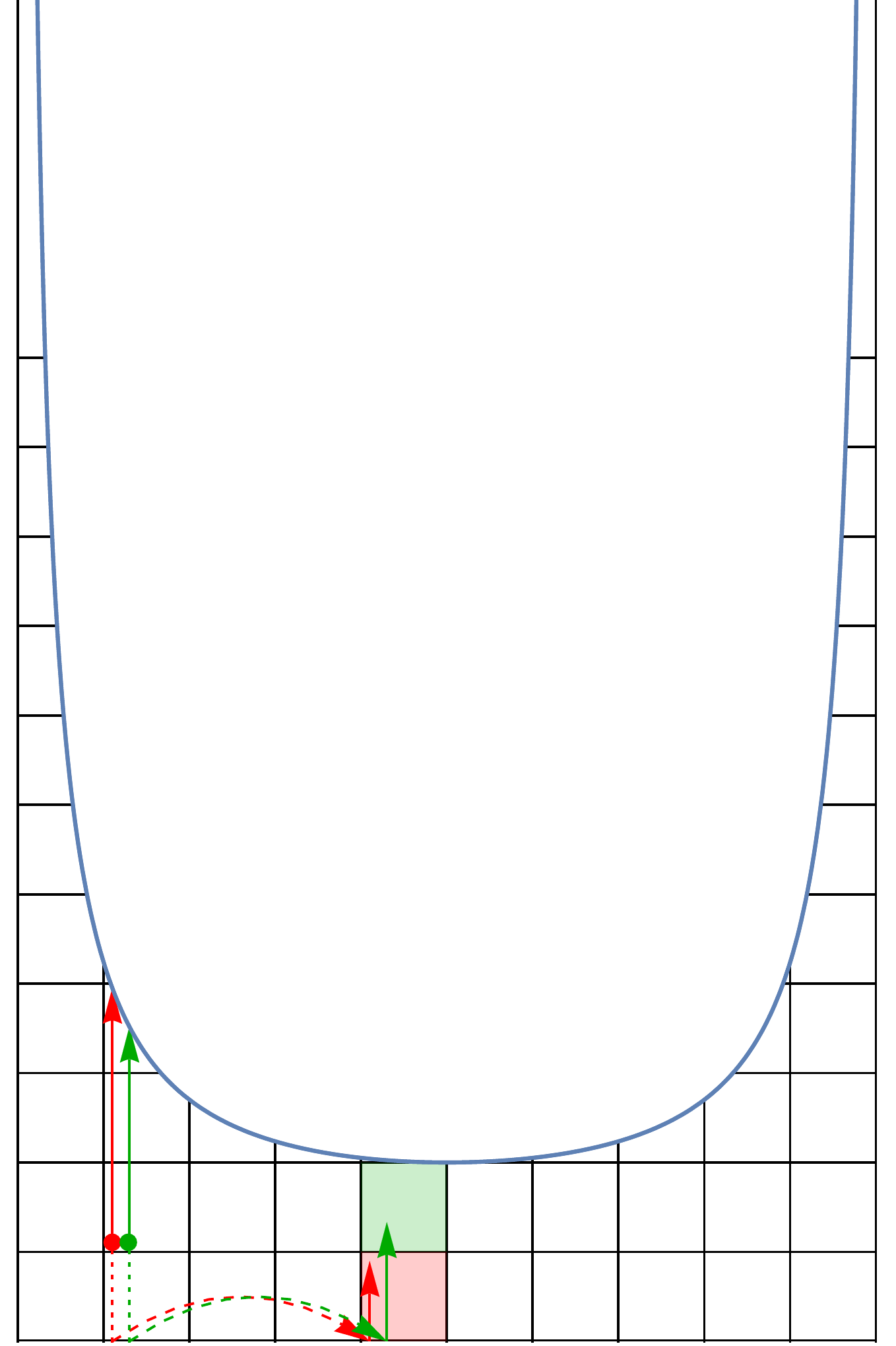}
\captionof{figure}{Vertical Separation, $f$ and $\varphi$ have moderate differences}
\label{fig:vert-sep}
\end{minipage}\hspace{.5in}%
\begin{minipage}{2.5in}
\centering
 \includegraphics[width=2.5in]{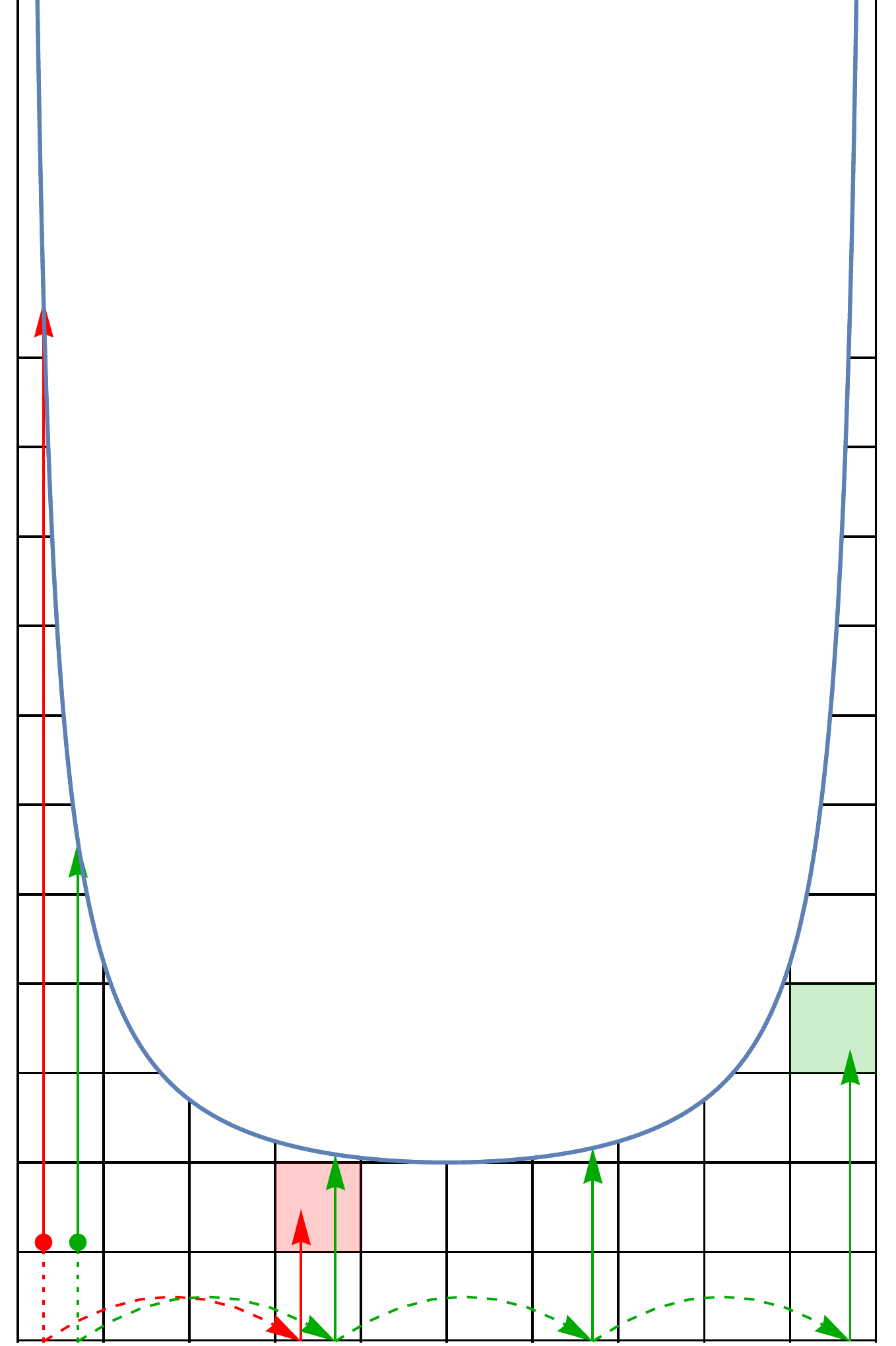}
\captionof{figure}{Horizontal Separation, $f$ and $\varphi$ have significant differences, the roof is hit a different number of times}
\label{fig:horiz-sep}
\end{minipage}
\end{figure}

Proposition \ref{techn} is then a collection of tools with which we can prove Proposition \ref{techn1}. Property 1. guarantees that our claims occur frequently throughout the orbit, property 2. guarantees that we stay out of the cusp (this gives a. and b.). Properties 3., 4., 5. and 6. serve to prove c.: property 3. is a way to deal with low values of $N$, which are covered by choosing our parameters sufficiently small. Property 6. guaranties that the orbits of $y$ and $y'$ don't come to close together so that they will split at time $\ll N$. 

The main content then lies in 4. and 5.  Recall that we wish to show that the amount of time spent in the set $A_j^{N,\xi_0}$ (for fixed $j$) is not long. Conclusion 5. deals with the case when horizontal separation occurs. The danger is the recurrence of circle rotations: if there is enough separation, the horizontal position of the orbits could become close after some time. Property 5. shows that for a time comparable to the inverse of the distance (up to a $\log$ factor) either there is {\it no} horizontal divergence, or the divergence is very large (see the constant $100$) compared to the original closeness (see Figure \ref{fig:horiz-sep}). In view of 5. the only way points can stay long in $A_j^{N,\xi_0}$ is when they move isometrically with respect to the horizontal direction. Property 4. then ensures that at a scale comparable to the inverse of the distance (up to a small power), the number of such occurences is small (in fact we have $1-\eta_0$ gain).  
This follows by the study of {\em Birkhoff sums} of $f'$, as we show later that for most of the times the growth of $f'$ is $n^{2-\eta}$, so points in 4. will split in the vertical direction (see Figure \ref{fig:vert-sep}).




\subsection{A proposition which implies Theorem \ref{thm:vwb}}
Let $\eta_0\ll\eta$.
\begin{proposition}\label{techn1} There exist  a set $B\subset \T^2\times M$, $\lambda(B)>9/10$, a {measurable} family of sets $(C_y)_{y\in M}$, $\mu(C_y)\geq 9/10$ and ${\frac{2^{\eta_0}-1}{2}}>\xi_0>0$ (arbitrary small) such that for every $(x,y)\in B$, $(x',y')\in B\cap \T^2\times C_y$ and  every $N\in\N$ there exists a set $U_N=U_N(x,y,x',y')\subset [0,N]$ such that
\begin{enumerate}[a.]
	 \item $|U_N|\geq \frac{9N}{10}$, 
    \item for every $n\in U_N$,  $y_n=K_{\vphi_n(x)}(y)\notin\{(y,s)\in M\;:\;s<\xi_0^{-1}\}$, 
	 \item for every $j\in \N$ we have
$$\left|U_N\cap A_j^{N,\xi_0}(x,y,x',y')\right|\leq \frac{N}{2^{j\eta_0}}.$$ 
\end{enumerate}
\end{proposition}
 Let us show how Proposition \ref{techn1} implies Theorem \ref{thm:vwb}.
\begin{proof}[Proof of Theorem \ref{thm:vwb}] 
Let $Q$ be the following partition of $M$: The set 
$\{(y,s)\in M\;:\;s\geq \xi_0^{-1}\}$ is one atom of $Q$; divide the (compact) set $\{(y,s)\in M\;:\;s< \xi_0^{-1}\}$
into atoms of diameter $<\xi_0$. Notice that 
by \eqref{anj}, for {$j\leq -1-\log_2\xi_0$},
$$
A_j^{N,\xi_0}(x,y,x',y')=\emptyset.
$$
Therefore, for every $(x,y)\in B$, $(x',y')\in B\cap \T^2\times C_y$ and  every $N\in\N$, we have using a., b. and c. and since $\xi_0<{\frac{2^{\eta_0}-1}{2}}$


\begin{multline*}
D^Q_N(x,y,x',y')\leq \frac{|U_N^c|}{N} + \frac{1}{N}\left|\{n\in U_N\;:\;  d(y_n,y'_n)<\xi_0 \}\right| \leq  \\
1/10+ \frac{1}{N}\sum_{j\geq 1-\log_2\xi_0}^{\infty}|U_N\cap A_j^{N,\xi_0}(x,y,x',y')| 
\leq 1/10+1/2<9/10 
\end{multline*}
and this finishes the proof.
\end{proof}
So it remains to prove Propostion \ref{techn1}.
\subsection{The main technical result; proof of Proposition \ref{techn1}}
The following proposition implies Proposition \ref{techn1}. Recall that $d$ is a metric on $M$ and we denote by $d_H$ and $d_V$ the distances on first and second coordinate respectively. Recall that $\eta_0\ll \eta$, $y_m:=K_{\vphi_m(x)}y$, $y'_m:=K_{\vphi_m(x')}y'$ and $R_m:=d_H(y_m,y'_m)^{-1}$.

\begin{proposition}\label{techn} There exist  a set $B\subset \T^2\times M$, $\lambda(B)>9/10$, a {measurable}  family of sets $(C_y)_{y\in M}$, $\mu(C_y)\geq 9/10$ and there exists $N'$ such that for every $N_3 \geq N'$ there exists $\xi_3>0$ {(small)} such that for every $(x,y)\in B$, $(x',y')\in B\cap \T^2\times C_y$ and  every $N\in\N$ there exists a set $U_N=U_N(x,y,x',y')\subset [0,N]$ such that
\begin{enumerate}
	 \item $|U_N|\geq \frac{9N}{10}$, 
    \item for every $n\in U_N$,  $y_n\notin\{(y,s)\in M\;:\;s<\xi_3^{-1}\}$, 
	 \item for every $j\in \N$ and $N\leq N_3$
$$U_N\cap A_j^{N_3,\xi_3}(x,y,x',y')=\emptyset;$$
    \item if $N\geq N_3$; for every $m\in U_N$ if $d(y_m,y_m')<\xi_3$,
 then 
\begin{multline*}
|\{n\in [-R_m^{1-10\eta},
R_m^{1-10\eta}]\;:
\;d_H(y_{n+m},y'_{n+m})=d_H(y_m,y'_m)\\
\text{ and }\;d_V(y_{n+m},y'_{n+m})<1
\}|<3R_m^{(1-10\eta)(1-\eta_0)};
\end{multline*}
\item If $N \ge N_3$, $m\in U_N$ and $d_H(y_m,y'_m)<\xi_3$ then for every $n\in [0,\frac{R_m}{\log^5R_m}]\cap \Z$ either $d_H(y_{n+m},y'_{n+m})=d_H(y_m,y'_m)$ or 
$$
d_H(y_{n+m},y'_{n+m})\geq 100d_H(y_m,y'_m);
$$
\item for every $N\geq N_3$
$$
\min_{0\leq i<N}d_H(y_i,y'_i)\geq \frac{1}{N\log^6N}\;\;\;\;\; i.e\;\;\max_{0\leq i<N}R_i\leq N\log^6N
$$
 \end{enumerate}
\end{proposition}

The proof of Proposition \ref{techn} is the most involved part of the paper, it will be given in separate subsection.

\begin{proof}[Proof of Proposition \ref{techn1}]
Notice that a. and b. in Proposition \ref{techn1} follow by 1. and 2. in Proposition \ref{techn}. Moreover if $N\leq N_3$, then c. in Proposition \ref{techn1} follows by 3. in Proposition \ref{techn}. So it remains to prove c. in Proposition \ref{techn1} asuming that 4., 5. and 6. in Proposition \ref{techn} hold. \\
\\

\begin{figure}
\begin{center}
\includegraphics[width=5in]{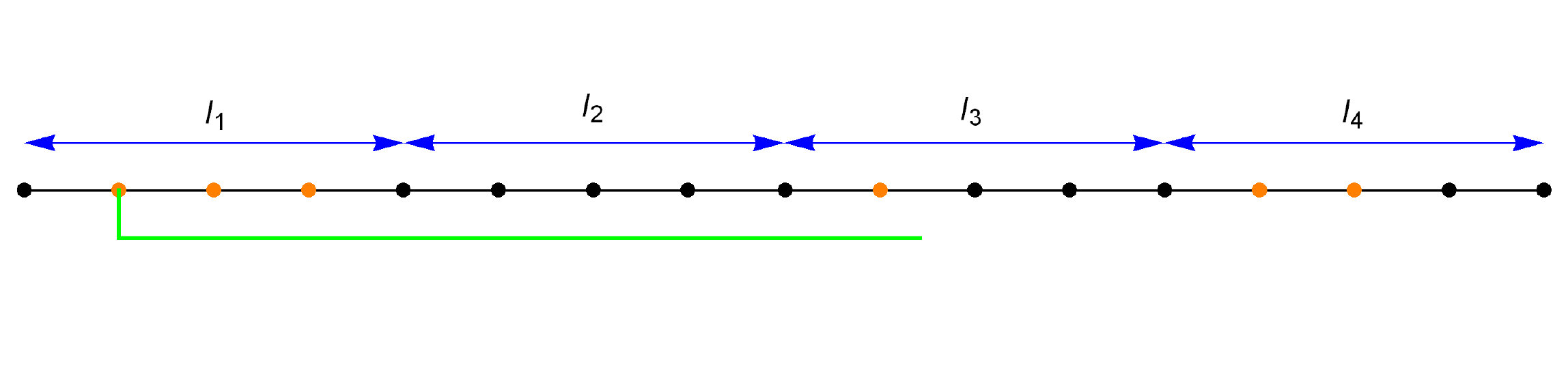}
\caption{Breaking up $[0,N]$}
\label{fig:prop41}
\end{center}
\end{figure}

Fix $j\in \N$ and assume that $U_N\cap A_j^{N,\xi_3}(x,y,x',y')\neq\emptyset$.  Divide the interval $[0,N]$ into intervals $I_1,...,I_k$ of length $2^{j(1-10\eta)}$. By 6. it follows that $2^{j}<N\log^6N$ and therefore $k>1$ assuming $N$ is large enough, i.e. $N_3$ large enough. 
Consider only those $I_i$, for which  $U_N\cap A_j^{N,\xi_3}(x,y,x',y')\cap I_i \neq \emptyset$.

 In Figure \ref{fig:prop41}, we see the scheme which we are about to pursue formally. We have already broken up the long interval
$[0,N]$ into smaller (blue) intervals $I_i$.
By dividing the interval $[0,N]$ into intervals of appropriate size, we may get estimates
on the times in $A_j^{N,\xi_3}$, represented by the orange points. The green interval represents the times after the first occurence of $A_j^{N,\xi_3}$ when we may apply the estimates of 4. and 5. of Proposition \ref{techn}. These guarantee to divergence to show that $A_j^{N,\xi_3}$ takes up very little
of $[0,N]$.

 Let $m$ be the first element of this intersection. Since $m\in U_N$ we obtain that if $n\in I_i$ then $n-m\leq 2^{j(1-10\eta)}<R_m^{(1-10\eta)}\leq \frac{R_m}{\log^5R_m}$ and hence by 5. either $d_H(y_n,y'_n)=
d_H(y_m,y'_m)$ or $d_H(y_n,y'_n)\geq 100d_H(y_m,y'_m)$.  If $n\in U_N\cap A_j^{N,\xi_3}(x,y,x',y')\cap I_i$ the second option is not possible by the definition of $A_j^{N,\xi_3}(x,y,x',y')$.

 So we obtained that:
\begin{multline*}
U_N\cap A_j^{N,\xi_3}(x,y,x',y')\cap I_i\subset \{n\in I_i\;:\;d(y_n,y'_n)<\xi_3\;\text{ and }\; 
d_H(y_n,y'_n)=
d_H(y_m,y'_m)\}.
\end{multline*}

%
Moreover, by 4. we get that
\begin{multline*}
 |\{n\in I_i\;:\;d(y_n,y'_n)<
\xi_3\;
\text{ and }\; 
d_H(y_n,y'_n)=
d_H(y_r,y'_r)\}|\leq 2^{j(1-10\eta)(1-\eta_0)+2}.
\end{multline*}
Therefore, summing over all $i\in \{1,...k\}$  we get 
$$
|U_N\cap A_j^{N,\xi_3}(x,y,x',y')| 
\leq \left[\frac{N}{2^{j(1-10\eta)}}+1\right]2^{j(1-10\eta)(1-\eta_0)}\leq
\frac{N}{2^{j\eta_0/2}}
$$
 

This gives c. in Proposition \ref{techn1}.
\end{proof}

The rest of the paper will be devoted for proving Proposition \ref{techn}. Properties 1., 2. will follow straightforward after we define the sets $B$ and $\{C_y\}_{y\in M}$. Properties 3., 5. and 6. are of intermediate level of difficulty, we will devote separate subsections for them. The heart of the problem is property 4. for which we devote a separate section.





\section{Construction of $B$ and $\{C_y\}_{y\in M}$}\label{bcy}
In this section we construct sets $B$ and $\{C_y\}_{y\in M}$ for Proposition \ref{techn}. Construction of $B$ is more involved
it uses precise estimates on Birkhoff sums for $\vphi, f,f'$ and we will conduct it in several subsections below. The construction of $C_y$ is much more straightforward, we will give it now.
\subsection{Construction of $\{C_y\}_{y\in M}$.} 
For $y\in M$ we have $y'\in C_y^i\subset M$ if 
$$
\inf_{t\in[-q_i\log q_i,q_i\log q_i]}d_H(y,K_ty')\geq \frac{1}{q_i\log^3q_i}.
$$
 
Notice that $\mu(C^i_y)\geq 1-\frac{2}{\log^{2}q_i}$.
Let 
$$
C^0_y:=\bigcap_{i\geq n_0}C^i_y,
$$
where $n_0\gg 1$ is such that $\mu(C^0_y)\geq 99/100$.  Define 
\begin{equation}\label{d:cy}
C_y=C^0_y\cap\{y'\in M\;:\; \min_{|t|<n_0^2}d_H(K_t(y'),y)\geq \frac{1}{n_0^3}\}
\end{equation}
Then $\mu(C_y) \geq 9/10$.
This way we defined a family of sets $\{C_y\}_{y\in M}$. We will later show that this family satisfies the assumptions of Proposition \ref{techn}. We will now construct the set $B$.
\subsection{Construction of $B$}
\paragraph{Scheme of the construction} The construction of $B$ brakes down into several steps. First we construct sets $E_0$ and $F_0$ for which the desired asymptotics for the hyperbolic part holds (see \eqref{yeg1} and \eqref{yegint}). Then we define the set $V$ (see Proposition \ref{mpr}) which gives correct asymptotics for $f'$, i.e. for $(x,y)\in V$ the derivative is large for most of the times (see the definition of $W_n$). This corresponds for spliting in vertical direction. We also have the set $S$ which allows to control the distance to singularity at scales $(q_n)$. This gives the set $G$ (see \eqref{gse}). The set $B$ is the set of points whose orbits visit $G$ with correct propotion.

\subsubsection{Birkhoff sums for $\vphi$}\label{sec:Ans}
Let us denote $\vphi_0=\int_{\T^2}\vphi \, d\lambda$ and $\vphi_n = \sum_{k=0}^{n-1} \varphi\circ A^k$, where $A$ is a hyperbolic toral automorphism. The following is an immediate consequence of the central limit theorem for H\"{o}lder continuous observables:

\begin{lemma}
For $\lambda$-almost every $x \in \T^2$,
\begin{equation}\label{lil}
\lim_{|n|\to+\infty}\frac{|\vphi_n(x)-n\vphi_0|}{|n|^{1/2}\log |n|}=0.
\end{equation}
\end{lemma}

Therefore, by Egorov's theorem, there exists $N_0\in \N$ and $E_0\subset \T^2$, $\lambda(E_0)>99/100$ such that for every $x\in E_0$ and every $N\in \Z$, $|N|\geq N_0$

\begin{equation}\label{yeg1}
|\vphi_n(x)-n\vphi_0|\leq |n|^{1/2}\log|n|.
\end{equation}
For simplicity we assume that $\vphi_0=1$. Then, by Egorov's theorem,
there exists a set $F_0$ and $N'$ such that for every $x\in F_0$ and $|n|\geq N'$
 \begin{equation}\label{yegint}
\vphi_n(x)\geq \frac{|n|}{2}.
\end{equation}

\subsubsection{Birkhoff sums for $f'$}\label{sec:Koch}
Recall that for $y\in M$, $y_0\in \T$ denotes the first coordinate of $y$. Let
$$
S_n:=\left\{y\in M\;:\;\bigcup_{t=-q_n\log q_n}^{q_n\log q_n}K_t(y)\notin [-\frac{1}{q_n\log^3q_n},\frac{1}{q_n\log^3q_n}]^f\right\}.
$$
Notice that $\mu(S_n)\geq 1- \frac{2}{\log^2q_n}$.
Let moreover, 
\begin{equation}\label{gse}
S:=\left(\bigcap_{n\geq n_1}S_n\right)\cap\{y\in M\;:\; y_0\notin[-n_1^{-2},n_1^{-2}]\},
\end{equation}
where $n_1\in \N$ is such that $\mu(S)>99/100$.

Let for $n\in \Z$,  $W_n:=\{y\in M \;: \;\abs{f'_n(y_0)}\geq |n|^{2-4\eta}\}$. 

We have the following proposition:

\begin{proposition}\label{mpr} Let $\delta=\frac{\eta^3}{1000}$. There exists $V\subset \T^2\times M$, $\lambda\times \mu(V)\geq 29/30$ and $N_1\in \N$ such that for every $(x,y)\in V$, $N\geq N_1$ we have
\begin{equation}\label{eq:m}
\left|\{i\in[-N,N]\;: y\notin W_{N(y,\vphi_i(x))}\;\}\right| \leq (2N)^{1-\delta}.
\end{equation}
\end{proposition}

We will give the proof of Proposition \ref{mpr} in a separate subsection. We have now defined all sets, which are needed in the definition of $B$ in Proposition \ref{techn}. 

Consider the set 
\begin{equation}\label{defg}
G:=V\cap (F_0\cap E_0)\times(S\cap\{(y,s)\in M\;:\; f(y)<10^4\}),
\end{equation}
where $E_0$ commes from \eqref{yeg1}, $S$ from \eqref{gse} and $V$ from Proposition \ref{mpr}. Then it follows that $\lambda\times \mu(G)\geq 19/20$. Therefore:\\
\begin{multline}
\text{ there exists a set }B_{erg}\subset \T^2\times M, \lambda\times \mu(B_{erg})\geq 19/20\\
\text{ and there exists }N_2\in \N\text{ such that for every }(x,y)\in B_{erg}\text{, and every }N\geq N_2\text{, we have } 
\end{multline}
\begin{equation}\label{yeg2}
 \frac{1}{N}\left|\{i\in[0,N-1]\;:\;A_\vphi^i(x,y)\in G\}\right|\geq 9/10.
\end{equation}

Define \begin{equation}\label{un}U_N(x,y,x',y'):=\{i\in[0,N-1]
\;:\;A_\vphi^i(x,y)\in G\},
\end{equation}
for $N\geq N_2$ and $U_N=[0,N]$ if $N\leq N_2$. Finally define $B$ by 
\begin{equation}\label{d:b}
B:=B_{erg}\cap\left(\T^2\times \left(S\cap\left\{y\in M\;:\;\bigcup_{t=-N_2^2}^{N_2^2}K_t(y)\notin [-N_2^{-4},N_2^{-4}]^f\right\}\right)\right). 
\end{equation}

We will show that the set $B$, sets $\{C_y\}_{y\in M}$ (see \eqref{d:cy}) and $U_N$ satisfy the assumptions of Proposition \ref{techn}. \\

\paragraph{Properties 1. and 2. in Proposition \ref{techn}.}\label{12} Notice that 1. and 2. follow by the definition of $U_N$ (see \eqref{yeg2}) and the set $B$ (if $\xi_0$ is chosen sufficiently small).\\
\\
So to complete the proof of Proposition \ref{techn}, we need to show 3., 4., 5., 6.

\subsection{Proof of 5. in Proposition \ref{techn}}\label{p5}
Property 5. is a straightforward consequence of the following general lemma. Recall that $y_n:=K_{\vphi_n(x)}y$, $y'_n:=K_{\vphi_n(x')}y'$ and $R_n:=d_H(y_n,y'_n)^{-1}$, $R_0:=d_H(y,y')^{-1}$.

%
%
%
%

\begin{lemma}\label{cons:dis} 
There exists $N_0\in \N$ such that for every $x,x'\in \T^2$ and $y,y'\in M$, if $R_0\geq N_0$ and if $n\in [0,R_0\log^{-5}R_0]\cap \Z$ either 
$d_H(y_n,y'_n)=d_H(y,y')$ or 
$$
d_H(y_n,y'_n)\geq 100d_H(y,y')
$$
\end{lemma}
\begin{proof} By diophantine assumptions on $\a$ we have
for every $m\geq 1$
\begin{equation}\label{dia}
\inf_{1\leq |s|\leq m}\|s\a\|\geq \frac{C(\a)}{m\log^2m}.
\end{equation}
Take $N_0$ such that $\log N_0\gg C(\a)^{-1}$.
For $n\in[0,R_0\log^5R_0]\cap \Z$  the first coordinates of $y_n$ and $y'_n$ are respectively $y_0+m_n\a$, $y_0'+r_n\a$ for some $m_n,r_n\in \N$. Since $|\vphi|<C$ and $f>c$, it follows that $m_n,r_n<C_0n$ for some constant $C_0$. Therefore we have either $m_n=r_n$ in which case $d_H(y_n,y'_n)=|y_0-y_0'|=d_H(y,y')$ or by \eqref{dia} and $|m_n-r_n|<2C_0n<2C_0R_0\log^{-5}R_0\leq R_0\log^{-4}R_0$ 
we get  
\begin{eqnarray*}
d_H(y_n,y'_n)&\geq&\|(m_n-r_n)\a\|-d_H(y,y')\geq\frac{C(\a)}{R_0\log^{-4}R_0\log^2(R_0\log^{-4}R_0)}-\frac{1}{R_0}\\
&=&\frac{1}{R_0}\left(\frac{C(\a)\log^4R_0}{\log^2(R_0\log^{-4}R_0)}-1\right)> \frac{100}{R_0}=100d_H(y,y').
\end{eqnarray*}
This finishes the proof.
\end{proof}

\subsection{Proof of 3. and 6. in Proposition \ref{techn}}\label{p36}
We have the following easy lemma:

\begin{lemma}\label{shor:bl} There exists $N_4\in \N$ such that for every $y\in M$, every $y'\in C_y$ (see \eqref{d:cy}) and every $N\in \Z$, $|N|\geq N_4$ we have
\begin{equation}\label{sb}
d_H(y, K_Ny')\geq \frac{1}{N\log^5N}. 
\end{equation}
\end{lemma}
\begin{proof} 
Notice that by \eqref{d:cy}, for sufficiently large $n$, we have
$$
\min_{-q_n\log q_n\leq s\leq q_n\log q_n} d_H(y, K_{s}y')\geq \frac{1}{q_n\log^{3}q_n}.
$$
Since $q_{n+1}<q_n\log^{11/10}q_n$, \eqref{sb} follows.
\end{proof}
\begin{corollary}\label{cor:blo} For every $(x,y)\in B $, $(x',y')\in B\cap (\T^2\times C_y)$ and $n\geq 0$ we have
\begin{equation}\label{sb2}
R_n^{-1}=d_H(y_n,y_n')\geq \min\left(\frac{1}{N_4^3},\frac{1}{n\log^6n}\right).
\end{equation}
\end{corollary}
\begin{proof} Let $y+m_n\alpha$ and $y'+r_n\alpha$ denote first cooridnates of respectively $y_n, y'_n$. Since $|\vphi|<C, f>c$ it follows that $m_n, r_n\leq C_0 n$ for some $C_0>0$. Therefore for $n\geq N_4$, by \eqref{d:b} and $y'\in C_y$, we have 
$$
d_H(y_n,y_n')= \|y+y'+(m_n-r_n)\alpha\|\geq \inf_{|s|\leq C_0 n}d_H(y,K_sy)\stackrel{sb}{\geq} \frac{1}{n\log^6n}.
$$  
\end{proof}
Therefore 3. and 6. in Proposition \ref{techn} follow by \eqref{sb2} if only we take $N_3\gg N_4$ and $\xi_3\ll N_4^{-1}$. 

\section{Proof of 4. in Proposition \ref{techn}}\label{sec.4}
In this section we assume that Proposition \ref{mpr} holds (we prove Proposition \ref{mpr} in Section \ref{p:mpr}). Recall that $y_n:=K_{\vphi_n(x)}y$, $y'_n:=K_{\vphi_n(x')}y'$ and $R_n:=d_H(y_n,y'_n)^{-1}$, $R_0:=d_H(y,y')^{-1}$. For $y\in M$ let $y_0$ denote the first coordinate of $y$.

By the definition of $U_N$ (see \eqref{un}) and $B$ (see \eqref{d:b}), property 4. of Proposition \ref{techn} is a strightforward consequence of the following lemma (by applying the lemma to $y=y_m$ and $y'=y'_m$): 

\begin{lemma}\label{cruc} For  every $(x,y), (x',y')\in G$ (see \eqref{defg}) such that $R_0>N_3$ we have 
\begin{multline*}
|\{n\in [-R_0^{1-10\eta},
R_0^{1-10\eta}]\;:
\;d_H(y_n,y'_n)=d_H(y,y')\\
\;\text{ and }\;
d_V(y_n,y'_n)<1
\}|<3R_0^{(1-10\eta)(1-\eta_0)}.
\end{multline*}
\end{lemma}

\begin{proof} 
By the definition of special flow every $n$ which belongs to the set satisfies $N(y,\vphi_n(x))=N(y',\vphi_n(x'))$ and we have
\begin{equation}\label{horsm}
 \left|\vphi_n(x)-\vphi_n(x')+
f_{N(y,\vphi_n(x))}(y_0)-f_{N(y,\vphi_n(x))}(y'_0)\right|<2.
\end{equation}
But since $x,x'\in E_0$, $|\vphi_n(x)-\vphi_n(x')|\leq 2|n|^{1/2}\log |n|$. Moreover, if $|n|>R_0^{9/10}$ and $y\in W_{N(y,\vphi_n(x))}$, then by Lemma \ref{sub2},  
$$
|f_{N(y,\vphi_n(x))}(y_0)-f_{N(y,\vphi_n(x))}(y_0')|\geq |n|^{3/4},
$$
so \eqref{horsm} does not hold. 
It remains to notice, that since $y\in V$, by Propostion \ref{mpr}  
$$
\left|\{n\in[-R_0^{1-10\eta},R_0^{1-10\eta}]\;:\; y \not\in W_{N(y,\vphi_n(x))}\}\right|\leq 2R_0^{(1-10\eta)(1-\eta_0)}.
$$
This finishes the proof.
\end{proof}
 
\begin{lemma}\label{sub2}Let $(x,y)\in G$ and $y'\in M$, $R_0>N_3$. \\
Then for every $n \in[-R_0^{(1-10\eta)},-R_0^{9/10}]\cup [R_0^{9/10},R_0^{(1-10\eta)}]$ such that $y\in W_{N(y,\vphi_n(x))}$ we have
$$ 
|f_{N(y,\vphi_n(x))}(y_0)-f_{N(y,\vphi_n(x))}(y'_0)|\geq |n|^{3/4}.
$$
\end{lemma}
\begin{proof} Let us conduct the proof for $n>0$, the case $n<0$ is analogous. Let $k\in \N$ be unique such that $\frac{1}{q_{k+1}}\leq d_H(y,y')=\|y_0-y'_0\|<\frac{1}{q_k}$, i.e. $q_k<R_0\leq q_{k+1}$. Since $\sup_{\T^2}|\vphi|<C$ and $\inf_\T f>c$, we have $N(y,\vphi_n(x))<C_0n$ for some constant $C_0$. By diophantine assumptions on $\a$ it follows that $n<{R_0^{1-10\eta}\leq q_{k+1}^{1-10\eta}}\leq q_k^{1-\frac{\delta}{2}}$ for some $\delta>0$. Since $y\in S$ it follows that 
$$
\sup_{0\leq i <N(y,\vphi_n(x))}d(y_0+i\a,0)\geq q_k^{-1+\frac{\delta}{4}}.
$$
Therefore, by $\|y_0-y'_0\|<\frac{1}{q_k}$, we have 
for $i=0,...,N(y,\vphi_n(x))$
$$
-i\a\notin [y_0,y'_0].
$$
So, for some $\th\in[y_0,y'_0]$.
\begin{equation}\label{byh}
|f_{N(y,\vphi_n(x))}(y_0)-f_{N(y,\vphi_n(x))}(y'_0)|\geq 
|f'_{N(y,\vphi_n(x))}(y_0)||y_0-y'_0|-
|f''_{N(y,\vphi_n(x))}(\th)||y_0-y'_0|^2
\end{equation}
But since $y\in W_{N(y,\vphi_n(x))}$ and by Sublemma \ref{sub1},
$$|f'_{N(y,\vphi_n(x))}(y_0)|\geq \left(N(y,\vphi_n(x))\right)^{2-4\eta}\geq  \frac{n^{2-4\eta}}{\log^{12} n}.
$$
Moreover by Sublemma \ref{znp} and since $|\vphi_n(x)|<C_0n$

$$|f''_{N(y,\vphi_n(x))}(\th)|\leq \left(\vphi_n(x)\right)^{3+2\eta}<n^{3+3\eta}.
$$

But $n\leq |y_0-y'_0|^{-1+10\eta}$ and so
$$
\frac{n^{2-4\eta}}{\log^6 n}|y_0-y'_0|\geq 2
n^{3+3\eta}|y_0-y'_0|^2.
$$
Therefore, in \eqref{byh} we have since $|y_0-y'_0|^{-1}\geq n\geq |y_0-y'_0|^{-9/10}$
$$
|f_{N(y,\vphi_n(x))}(y_0)-f_{N(y,\vphi_n(x))}(y'_0)|\geq 
1/2|f'_{N(y,\vphi_n(x))}(y_0)||y_0-y'_0|\geq |y_0-y'_0|^{-3/4}\geq n^{3/4}.
$$

This finishes the proof.
\end{proof}

\begin{sublemma}\label{sub1} If $(x,y)\in G$ then for $|n|\geq N_3$ 
$$N(y,\vphi_n(x))\geq \frac{|n|}{\log^6|n|}.$$ 
\end{sublemma}
\begin{proof} Assume that $n>0$, the proof in case $n<0$ is analogous. Since $(x,y)\in G$, we know that $x\in F_0$. Hence for $n\geq N_3$, by \eqref{yegint} we know that $\vphi_n(x)\geq \frac{n}{2}$. Therefore, by the definition of special flow, we have
$$
\frac{n}{2}\leq \vphi_n(x)<f_{N(y,\vphi_n(x))+1}.
$$ 
Let $k\in \N$ be unique such that $q_k\log q_k\leq N(y,\vphi_n(x))<q_{k+1}\log q_{k+1}$. Then since  $y\in S$ (see \eqref{gse}), we know that
$$
\min_{0\leq i<N(y,\vphi_n(x))+1}d_H(y_0+i\alpha,0)\geq \frac{1}{q_{k+1}\log^{4}q_{k+1}}.
$$

Therefore, and by \eqref{koks0} in Lemma \ref{koksi}, we have
$$
f_{N(y,\vphi_n(x))}\leq 10q_{k+1}\leq N(y,\vphi_n(x))\log^3 N(y,\vphi_n(x)). 
$$
So $\frac{n}{2}< N(y,\vphi_n(x))\log^3 N(y,\vphi_n(x))$ and this finishes the proof. 
\end{proof}

\begin{sublemma}\label{znp} Fix $y\in S$ and $t\in \R$. For every $\th\in \T$ such that $\|\th-y_0\|<\frac{1}{|t|\log^3|t|}$ we have 
$$|f''_{N(y_0,t)}(\th)|<t^{3+3\eta}.$$
\end{sublemma}
\begin{proof}We will conduct the proof in the case $t>0$, the oposite case is analogous.
By Lemma \ref{koksi}  (using the estimates in \eqref{koks0} and \eqref{koks2}) we have
$$
|f''_{N(y_0,t)}(\th)|\leq |f_{N(y_0,t)}(\th)|^{3+2\eta}.
$$
 Therefore it is enough to show that
$$
|f_{N(y_0,t)}(\th)|\leq t\log^{4}t.
$$
Let $k\in \N$ be unique such that $q_k\leq N(y_0,t)<q_{k+1}$.
Notice that since $y\in S$ and $\|\th-y_0\|<\frac{1}{t\log^3t}$ we know that 
$$
\min_{0\leq i<N(y_0,t)}d(\th+i\a,0)\geq \frac{1}{2q_{k+1}\log^3 q_{k+1}}.
$$
Therefore, by Lemma \ref{koksi} and diophantine assumptions on $\a$, and $N(y_0,t)\leq (\inf_\T f)t$
$$
f_{N(y_0,t)}(\th)< 2q_{k+1}\leq N(y_0,t)\log^3N(y_0,t)\leq t\log^4t.
$$
This finishes the proof.
\end{proof}

\section{Growth of the derivative, proof of Proposition \ref{mpr}}\label{p:mpr}
Recall that for $y\in M$, $y_0\in \T$ denote the first coodinate of $y$ and that $W_n=\{y\;:\; |f'_n(y_0)|\geq |n|^{2-4\eta}\}$.

Proposition \ref{mpr} will follow from the proposition and lemma below:

\begin{proposition}\label{prob} We have 
$$
\lim_{|N|\to +\infty }|N|^{-1+\frac{\eta}{10}}\sum_{i=0}^{N-1}\chi_{W_i^c}(y)=0.
$$
\end{proposition}

\begin{lemma}\label{con} Fix $(x,y)\in \T^2\times M$ and $M\in \Z$. If  
$|\vphi_M(x)-M|\leq M^{2/3}$, $|f_M(y)-M|<M^{1-\frac{\eta}{2}}$ and 

\begin{equation}\label{ase}\left|\{i\in[0,M]\;:\; y\notin W_i\}\right|<|M|^{1-\frac{\eta}{10}},
\end{equation}
 then 
\begin{equation}\label{eq:con}
\left|\{i\in[0,M]\;:\; y\notin W_{N(y,\vphi_i(x))}\}\right|<|M|^{1-\frac{\eta^3}{1000}}.
\end{equation}
\end{lemma}

\begin{proof}[Proof of Proposition \ref{mpr}]
By Proposition \ref{prob} and Egorov theorem, there exists a set $V_0\in M$, $\mu(V_0)\geq 99/100$ and $N_0\in \N$ such that for every $y\in V_0$ and $|M|\geq N_0$ \eqref{ase} is satisfied.   Define $V:=(\T^2\times V_0)\cap (E_0\times M)\cap (\T^2\times S)$ (see \eqref{yeg1} and \eqref{gse}) and $N_1$ sufficiently large. Then since $x\in E_0$,  $|\vphi_N(x)-N|\leq N^{2/3}$. Moreover, 
let $s$ be unique such that $q_s\leq N<q_{s+1}$
 since $y\in S$, for $s\geq n_0$
$$
\{x+i\a\}_{i=-q_{s+1}\log q_{s+1}}^{q_{s+1}\log q_{s+1}}\notin [-\frac{1}{q_{s+1}\log^3 q_{s+1}},\frac{1}{q_{s+1}\log^3 q_{s+1}}].
$$
Hence, by \eqref{koks0} in Lemma \ref{koksi}
$$|f_N(y)-N|< N^{1-\frac{\eta}{2}}.
$$
So for any $(x,y)\in V$ and $|N|\geq N_1$ sufficiently large, the assumptions of Lemma \ref{con} are satisfied. Therefore \eqref{eq:con} holds for $y$ and $M$. By taking $M=N, -N$ we get that also \eqref{eq:m} holds . The proof of Proposition \ref{mpr} is thus finished.  
\end{proof}

Therefore it remains to prove Proposition \ref{prob} and Lemma \ref{con}.

\paragraph{Proof of Proposition \ref{prob}}
We will first prove the following lemma:
\begin{lemma}\label{smd}
There exists a constant $C>0$ such that for every $n\in \Z$ 
$$
\mu(W^c_n)<C|n|^{-\frac{3\eta}{2}}.
$$
\end{lemma}
\begin{proof} Let us conduct the proof in the case $n>0$ the case $n<0$ is analogous. Let $s\in \N$ be unique such that $q_s\leq n<q_{s+1}$. Let $I=(a,b]$ be any interval in the partition $\cI$ of  $\T$ given by $\{-i\a\}_{i=0}^{n-1}$. It follows by \eqref{asu},\eqref{asu2}, \eqref{asu3}, that $f_n$ is $C^2$ on $I$. Moreover $\lim_{x\to b^-}f'_n(x)=+\infty$ and $\lim_{x\to a^+}f'_n(x)=-\infty$. Hence there exists $x_I\in I$ such that $f'_n(x_I)=0$. Then  for $x\in I$ there exists $\th_x\in I$ such that
\begin{equation}\label{c2f}
|f'_n(x)|=|f'_n(x)-f'_n(x_I)|=|f''_n(\th_x)||x-x_I|.
\end{equation}
Moreover, by Lemma \ref{koksi} $|f''_n(\th_x)|\geq q_s^{3-\eta}\geq n^{3-2\eta}$ (the last inequality by diophantine assumptions on $\a$). Let $I_{bad}:=[-\frac{1}{n^{1+2\eta}}+x_I,x_I+\frac{1}{n^{1+2\eta}}]
$.
Then by \eqref{c2f}
$$
 (W^c_n\cap I)\subset I_{bad}.
$$

So
$$ 
W^c_n\subset \bigcup_{I\in \cI}I_{bad}^f;
$$
and therefore  
$\mu(W^c_n)\leq Cn^{-\frac{3\eta}{2}}$.
\end{proof}
Now we can give the
\begin{proof}[Proof of Proposition \ref{prob}]
The proof uses some ideas of the proof of Theorem 1. in \cite{Lyo}.
We will use the following simple lemma:
\begin{lemma}\label{lio}\cite{Lyo} If $Y_n$ are random variables such that $\sum_{n\geq 1} \|Y_n\|_2^2<\infty$, then $Y_n\to 0$ a.s.
\end{lemma}

Let us denote $X_i:=\chi_{W^c_i}$. 
Notice that by Lemma \ref{smd}, we have   
$$
\left\|\frac{1}{|M|^{1-\frac{\eta}{10}}}\sum_{i=0}^{M-1}X_i\right\|_2^2\leq M^{-\frac{4\eta}{5}}.
$$
Therefore there exists a sequence $(N_k)$, $|N_{k+1}-N_k|\leq |N_k|^{1-\frac{\eta}{5}}$ such that
$$
\sum_{k\geq 1}\left\|\frac{1}{|N_k|^{1-\frac{\eta}{5}}}\sum_{i=0}^{N_k-1}
X_i\right\|_2^2< +\infty.
$$
By Lemma \ref{lio} we get $\frac{1}{|N_k|^{1-\frac{\eta}{5}}}\sum_{i=0}^{N_k-1}X_i\to 0$ a.s. 
as $k\to +\infty$. Let $k\in \N$ be unique such that $N_k\leq M<N_{k+1}$. Then
\begin{multline*}
\left|\frac{1}{|M|^{1-\frac{\eta}{10}}}\sum_{i=0}^{M-1}X_i\right|\leq 
\left|\frac{1}{|N_k|^{1-\frac{\eta}{10}}}\sum_{i=0}^{N_k-1}X_i\right|+\max_{0\leq s\leq N_{k+1}-N_k}\left|\frac{1}{|N_k|^{1-\frac{\eta}{10}}}
\sum_{i=N_k+1}^{N_k+s}X_i\right|\leq\\
\left|\frac{1}{|N_k|^{1-\frac{\eta}{10}}}\sum_{i=0}^{N_k-1}X_i\right|+
\frac{|N_{k+1}-N_k|}{|N_k|^{1-\frac{\eta}{10}}},
\end{multline*}
which finishes the proof since $|N_{k+1}-N_k|<|N_k|^{1-\frac{\eta}{5}}$.
\end{proof}

\paragraph{Proof of Lemma \ref{con}}

\begin{proof} 
Without loss of generality, we may assume $\vphi$ is strictly positive (we may replace $\vphi$ with a cohomologous function which is strictly positive, if necessary). Denote $\delta=\eta/10$.
Let $V_M:=\{t\in [-M,M]\;:\; y\not\in W_{N(y,t)}\}$. We will show the following:
\begin{equation}\label{lamt}
\lambda(V_M)<M^{1-2\delta^3}.
\end{equation}
Notice that \eqref{eq:con} follows by \eqref{lamt}. Indeed, we have
$$\{i\in[0,M]\;:\; y\not\in W_{N(y,\vphi_i(x))}\}\subset \{i\in [0,M]\;:\; \vphi_i(x)\not\in V_M\}\cup \{i\in [0,M]\;:\; |\vphi_i(x)|\geq M\}.$$
The set $V_M$ is a union of disjoint intervals of length $\geq \inf_\T f$ (since $N(y,t)$ is piecewise constant). For each such interval $I$ we have 
$$
\left|\{i\in [0,M]\;:\; \vphi_i(x)\in I\}\right|\leq c|I|+1
$$
(for  a constant $c$ such that $\vphi>\frac{1}{2c}$). So by \eqref{lamt}
$$\left|\{i\in [0,M]\;:\; \vphi_i(x)\not\in V_M\}\right|\leq cM^{1-2\delta}.
$$
Moreover, since $|\vphi_M(x)-M|\leq M^{2/3}$ and $\vphi>\frac{1}{2c}$ we have
$$
\left|\{i\in [0,M]\;:\; \vphi_i(x)\geq M\}\right|\leq cM^{2/3}.
$$
This gives \eqref{eq:con}. So it remains to show \eqref{lamt}.\\
\\
Since $|f_M(x)-M|\leq M^{1-4\delta}$ and $\inf_\T f>0$ it is enough to show that for 
$$V_{f,y}:=\{t\in [0,f_M(y)]\;:\;y\in W_{N(y,t)}\}
$$ 
we have 
$$
\lambda(V_{f,y})<M^{1-\delta^3}. 
$$
Divide the interval $[0,f_M(x)]$ into intervals of length $q_k$, where $q_k$ is the denominator closest to $M^{\delta^2}$ (notice that the endpoints of the intervals are integers).  Denote these intervals by $[N_j,N_{j+1}]$, $j=0,..., \left[\frac{f_M(x)}{q_k}\right]-1$. Notice that by \eqref{ase}, for  at least  $\left[\frac{f_M(y)}{q_k}\right]-M^{1-2\delta}$ of j's we have
\begin{equation}\label{coh}
[N_j,N_{j+1}]\cap \{i\in[0,M]\;:\; y\in W_i\}=\emptyset
\end{equation}
For any such $j$ let $T_j$ be such that $N(y,T_j)=N_j$. Then by \eqref{coh}, 
$$
[T_j,T_{j+1}]\cap V_{f,y}=\emptyset.
$$
By definition $N_{j+1}-N_j=q_k$. Therefore and by \eqref{koks0}, we have 
$$|T_{j+1}-T_j|\geq f_{N(y,T_{j+1}-T_j)}(y)=f_{N_{j+1}-N_j}(y)=f_{q_k}(y)\geq  q_k-2q_k^{1-\eta}.
$$

So  finally
$$
\lambda(V_{f,y})\leq f_M(y)- \left(\left[\frac{f_M(y)}{q_k}\right]-M^{1-2\delta}\right)(q_k-2q_k^{1-\eta})\leq
M^{1-\delta^3},
$$
which finishes the proof.
\end{proof}

\section*{Appendix: Proof of Lemma \ref{generatingpartition}}

By Krieger generator theorem and inducing as necessary we may assume that $K$ is the special flow over a subsystem (with respect to some invariant measure) of a full shift $S$ on $2$ symbols with a roof function $f$ such that $\left|f-\int f\right|<\epsilon$ for some $\epsilon$ small enough (say $\epsilon<\frac{\int f}{100}$). That is, there is a measure $\nu_0$ on $\Sigma_2$, a function $f:\Sigma_2\to(0,\infty)$, $\left|f-\int fd\nu_0\right|<\epsilon$ for $\epsilon<\frac{\int fd\nu_0}{100}$, such that $$M=\{(w,s)\;:\;w\in\Sigma_2\;\mbox{and}\; 0\leq s<f(w)\} \;\;\;\;\;\;\;\;\mbox{and}\;\;\;\;\;\;\;\;K_t(w,s)= (w,t+s)$$ where we use the identification as in formula (\ref{FlowSpace}). As in formula (\ref{D-C}) we get for $(w,s)\in M$,

$$K_t(w,s) = (S^{N(w,s,t)}(w) , t+s-   f_{N(w,s,t)} (w))$$ 
   where $N(w,s,t)$ is the unique integer such that 

\begin{equation*}
0 \leq  t+s-   f_{N(w,s,t)} (w) \leq f(S^{N(w,s,t)}(w)),\end{equation*}    
and
$$
f_n(y)=\left\{\begin{array}{ccc}
f(w)+\ldots+f(S^{n-1}(w)) &\mbox{if} & n>0\\
0&\mbox{if}& n=0\\
-(f(S^n(w))+\ldots+f(S^{-1}(w)))&\mbox{if} &n<0.\end{array}\right.$$

Inducing again if necessary we may assume that $\int f \, d\nu_0$ is large enough so that for every $x\in\Sigma_d$ and $w\in \Sigma_2$, $|\varphi(x)|<f(w)$. This implies that 
\begin{equation}
\label{shortjump}
|N(w,s,\varphi(x))|\leq 1\;\;\;\;\;\;\;\mbox{and}\;\;\;\;\;\;\;|N(w,s,-\varphi(S^{-1}x))|\leq 1
\end{equation}
for every $(x,(w,s))\in\Sigma_d\times M$. Moreover, if $\varphi(x)\geq 0$ then $N(w,s,\varphi(x))\geq 0$ and if $\varphi(x)\leq 0$ then $N(w,s,\varphi(x))\leq 0$. Let $\hat P_2$ be the partitions into cylinders of $\Sigma_2$ centered at $0$ and length $3$ (one to the past and one to the future) and let $Q$ be the induced partition on $M$ with time atoms of size $1/H$ for $H$ large but $\frac{1}{H}>10\epsilon$, i.e. let $H\geq 0$ be large and $L\geq 0$ be integers such that $\frac{L}{H}<\int fd\nu_0-\epsilon<\int fd\nu_0+\epsilon<\frac{L+1}{H}$ (we assume here that $L$ is also very large). Then $(w,s)$ is in the same atom of $(w',s')$ if and only if the $-1, 0, 1$ positions of $w$ and $w'$ coincide and $s,s'\in[\frac{i}{H},\frac{i+1}{H})$ for $i\leq L$ or $s,s'\geq \frac{L}{H}$, in particular, if $(w,s)$ is in the same atom as $(w',s')$ then $|s-s'|<1/H$.

Assume now that $(x,(w,s))$ and $(x',(w',s'))$ are in the same atom of $\displaystyle \bigvee_{i=-\infty}^\infty T^i(P_d\times Q)$. Then $x=x'$ since $P_d$ is the partition into cylinders. Furthermore, since $T$ is a skew-product, for $i>0$ 
\begin{eqnarray*}
K_{\varphi_i(x)}(w,s)&=:&(w^{(i)},s^{(i)})=K_{\varphi(S^{i-1}x)}(w^{(i-1)},s^{(i-1)})\\
&=&(S^{N(w^{(i-1)},s^{(i-1)},\varphi(S^{i-1}x))}w^{(i-1)},s^{(i-1)})
\end{eqnarray*} lies in the same atom as 
\begin{eqnarray*}
K_{\varphi_i(x)}(w',s')&=:&(w'^{(i)},s'^{(i)})=K_{\varphi(S^{i-1}x)}(w'^{(i-1)},s'^{(i-1)})\\
&=&(S^{N(w'^{(i-1)},s'^{(i-1)},\varphi(S^{i-1}x))}w'^{(i-1)},s'^{(i-1)}),
\end{eqnarray*}
and similarly with $i<0$. By the choice of $\epsilon$ and $H$ and equation (\ref{shortjump}) (and its comment below), we have that for every $i$, if $(w,s)$ and $(w',s')$ lie in the same atom and $x\in\Sigma_d$ then $$|N(w,s,\varphi(x))-N(w',s',\varphi(x))|\leq 1.$$ Hence, in case $\varphi(x)\geq 0$, either $$N(w,s,\varphi(x))=N(w',s',\varphi(x))\;\;\;\;\;\mbox{or}\;\;\;\;\; N(w,s,\varphi(x))=1,\;N(w',s',\varphi(x))=0$$ or viceversa. 

\noindent{\bf Claim.} 
We use the notations above. If 
$(w,s)$ and $(w^{(1)},s^{(1)})$ are on the same atom of $Q$ as $(w',s')$ and $(w'^{(1)},s'^{(1)})$ respectively then $$N(w,s,\varphi(x))=N(w',s',\varphi(x)).$$


\begin{proof}[Proof of Claim.]
Assume that $\varphi(x)>0$, the case $\varphi(x)<0$ is analogous. Assume by contradiction that $$N(w,s,\varphi(x))=1,\;\;\;\mbox{and}\;\;\;N(w',s',\varphi(x))=0,$$ i.e. $$0\leq\varphi(x)+s-f(w)<f(Sw)\;\;\;\;\;\mbox{and}\;\;\;\;\; 0\leq \varphi(x)+s'<f(w').$$ Then $s^{(1)}=\varphi(x)+s-f(w)$ and $s'^{(1)}=\varphi(x)+s'$, so $s'^{(1)}-s^{(1)}=s'-s+f(w)$. Hence $$|s'^{(1)}-s^{(1)}|\geq f(w)-|s'-s|\geq\int f \, d\nu_0-\epsilon-\frac{1}{H}\geq \frac{L-1}{H}-\epsilon\geq \frac{2}{H}$$ contradictiong the choice of $\epsilon$, $L$ and $H$.
\end{proof}
Using the claim we obtain that $$N(w^{(i)}, s^{(i)}, \varphi(S^ix))=N(w'^{(i)}, s'^{(i)}, \varphi(S^ix))$$ for every $i$. Since the jumps of the $N$ are at most $1$ this implies that $w=w'$ and $s,s'$ are on the same time atom. The next claim solves the problem. 

\noindent{\bf Claim.}
There is a set of full measure $\tilde X$ so that if $(x,(w,s))$ and $(x',(w',s'))$ are on the same atom of $\displaystyle \bigvee_{i=-\infty}^\infty T^i(P_d\times Q)\cap \tilde X$ then $x=x', w=w$ and $s=s'$. 

\begin{proof}[Proof of Claim] We already showed that $x=x'$ and $w=w'$. Assume that $(x,(w,s))$ and $(x,(w,s'))$ are on the same atom.

$K_{\varphi_i(x)}(w,s)=(S^{N_i}w,\varphi_i(x)+s-f_{N_i}(w))$ where $N_i:=N(w,s,\varphi_i(x))$. By the previous claim we know that $N(w,s,\varphi_i(x))=N(w,s',\varphi_i(x))$ hence we get that $s^{(i)}=\varphi_i(x)+s-f_{N_i}(w)$ and $s'^{(i)}=\varphi_i(x)+s'-f_{N_i}(w)$.

By ergodicity of $T$, there is a set of full measure $\tilde X$ such that if $(x,(w,s))\in\tilde X$ and $\delta>0$ then there is $i$ such that $s^{(i)}\in\left(\frac{1}{H}-\delta,\frac{1}{H}\right)$. Assume by contradiction that $(x,(w,s))$ and $(x,(w,s'))\in\tilde X$ and let $\delta=\frac{s'-s}{2}>0$. Then $s'^{(i)}=s'-s+s^{(i)}>s'-s+\frac{1}{H}-\delta=\frac{1}{H}+\delta$ contradicting $s'^{(i)}$ and $s^{(i)}$ are on the same atom. 
\end{proof}

\bibliographystyle{plain}
\bibliography{skew-prods}

\end{document}